\documentclass[11pt]{amsart}

\usepackage[usenames]{color}
\usepackage{amssymb}
\usepackage{latexsym}
\usepackage{graphics} 
\usepackage{graphicx}
\usepackage{verbatim}

\DeclareSymbolFont{AMSb}{U}{msb}{m}{n} 
\DeclareMathSymbol{\N}{\mathbin}{AMSb}{"4E}
\DeclareMathSymbol{\Z}{\mathbin}{AMSb}{"5A}
\DeclareMathSymbol{\R}{\mathbin}{AMSb}{"52}
\DeclareMathSymbol{\Q}{\mathbin}{AMSb}{"51}
\DeclareMathSymbol{\I}{\mathbin}{AMSb}{"49}

\newtheorem{theorem}{Theorem}

\newtheorem{prop}{Proposition}
\newtheorem{corollary}{Corollary}
\newtheorem{lemma}{Lemma}

\newtheorem{remark}{Remark}

\newtheorem*{mainprop}{Proposition~\ref{mainprop}}
\newtheorem*{kobayashi}{Kobayashi's Theorem}
\newtheorem*{paredJSJ}{Characteristic Decomposition Theorem for Pared Manifolds}
\newtheorem*{paredThurston}{Thurston's Theorem for Pared Manifolds}
\newtheorem*{ikeda}{Ikeda's Theorem}
\theoremstyle{definition} 
\newtheorem{definition}[theorem]{Definition}

\newtheorem{q}[theorem]{Question}

\newcommand{\bi}{\begin{itemize}}
\newcommand{\ei}{\end{itemize}}
\newcommand{\be}{\begin{enumerate}}
\newcommand{\ee}{\end{enumerate}}

\title{Intrinsic Chirality of Graphs in 3-manifolds}
  \date{\today}
\author[E. Flapan, H.N.\ Howards]{Erica Flapan, Hugh Howards}

    
    \address{Department of Mathematics, Pomona College, Claremont, CA 91711, USA}

\address{Department of Mathematics, Wake Forest University, Winston-Salem, NC 27109, USA}

\thanks{The first author was supported in part by NSF Grant DMS-1607744}

\begin{document}

    \begin{abstract}The main result of this paper is that for every closed, connected, orientable, irreducible 3-manifold $M$, there is an integer $ n_M$ such that if $\gamma$ is a graph with no involution and a 3-connected minor $\lambda$ with $\mathrm{genus}(\lambda)>n_M$, then every embedding of $\gamma$ in $M$ is chiral. By contrast, the paper also proves that for every graph $\gamma$, there are infinitely many closed, connected, orientable, irreducible 3-manifolds $M$ such that some embedding of $\gamma$ in $M$ is pointwise fixed by an orientation reversing involution of $M$.\end{abstract}

    \maketitle

\section{Introduction}
 
	
	We say that a graph $\Gamma$ embedded in a $3$-manifold $M$ is {\it achiral}, if there is an orientation reversing homeomorphism $h$ of $M$ leaving $\Gamma$ setwise invariant.  If such an $h$ exists, we say that it is a homeomorphism of the pair $(M,\Gamma)$.  If no such homeomorphism exists, we say that the embedded graph $\Gamma$ is {\it chiral}.  
	
	We can think of a knot as an embedding of a circular graph with some number of vertices in $S^3$.  Independent of the number of vertices on the graph, some embeddings of it are chiral while others are achiral.  By contrast, there exist graphs which have the property that all of their embeddings in $S^3$ are chiral.  Such a graph is said to be {\it intrinsically chiral} in $S^3$ because its chirality depends only on the intrinsic structure of the graph and not on the extrinsic topology of an embedding of the graph in $S^3$.  The following theorem provides a method of constructing graphs that are intrinsically chiral in $S^3$.
	
\begin{theorem}
\cite{fl4}	
\label{IC} Every non-planar graph $\gamma$ with no order $2$ automorphism is intrinsically chiral in $S^3$.
\end{theorem}

Since chirality is defined for graphs embedded in any 3-manifold, it is natural to  ask whether a graph which is intrinsically chiral in $S^3$ would necessarily be intrinsically chiral in other 3-manifolds.  Our first result shows that no graph can be intrinsically chiral in every 3-manifold, or even in every ``nice'' $3$-manifold.   

\begin{theorem}   For every graph $\gamma$, there are infinitely many closed, connected, orientable, irreducible 3-manifolds $M$ such that some embedding of $\gamma$ in $M$ is pointwise fixed by an orientation reversing involution of $M$.
\end{theorem}

Theorem~\ref{theorem1} can be thought of as a generalization of the fact that every planar graph has an embedding in $S^3$ which is pointwise fixed by a reflection of $S^3$.

On the other hand, our main result is a generalization of Theorem \ref{IC}, which shows that for any given ``nice'' 3-manifold $M$, there are infinitely many graphs with no automorphism of order $2$ which are intrinsically chiral in $M$.  In particular, let $M$ be a closed, connected, orientable, irreducible 3-manifold.  Then by Kneser-Haken finiteness \cite{KH}, $M$ contains at most finitely many non-parallel disjoint incompressible tori.  Let $N_M$ denote either the maximal possible number of non-parallel disjoint incompressible tori in $M$ or 2, whichever is larger.  Now we define $n_M=\mathrm{dim} _{\mathbb{Z}_2} (H_1 (M,\mathbb{Z}_2))+N_M$.  We will refer to the constants $n_M$ and $N_M$ in the statements and proofs of Theorem~\ref{ICManifolds} and Proposition~\ref{mainprop}.
  
Finally, we make use of the following definitions from graph theory.  A graph $\gamma$ is said to be {\it $3$-connected}, if at least three vertices together with the edges containing them must be removed to disconnect $\gamma$ or reduce $\gamma$ to a single vertex.  A graph $\gamma_1$ is said to be a {\it minor} of a graph $\gamma_2$, if $\gamma_1$ can be obtained from $\gamma_2$ by deleting and/or contracting some number of edges.  Our main result is as follows.

\begin{theorem}\label{ICManifolds}  For every closed, connected, orientable, irreducible 3-manifold $M$, any graph with no order $2$ automorphism which has a 3-connected minor $\lambda$ with $\mathrm{genus}(\lambda)>n_M$ is intrinsically chiral in $M$.
\end{theorem}

For example, let $M$ be a closed, connected, orientable, irreducible Seifert fibered space whose base surface $S$ has genus $g\geq 2$.  Then by using a standard pants decomposition argument (see for example \cite{Ha2}), it can be shown that the closed surface $S$ contains at most $3g -3$ disjoint, non-parallel essential circles.  This implies that $M$ contains at most $3g-3$ disjoint, non-parallel incompressible vertical tori.  Since $M$ is  Seifert fibered and $g\geq 2$, all incompressible tori in $M$ are vertical.  Hence $M$ has at most $3g-3$ disjoint, non-parallel incompressible tori.  Now let $n_M=\mathrm{dim} _{\mathbb{Z}_2} (H_1 (M,\mathbb{Z}_2))+3g-3$.  Then by Theorem \ref{ICManifolds}, any graph with no $2$ automorphism which has a 3-connected minor $\lambda$ with $\mathrm{genus}(\lambda)>n_M$ is intrinsically chiral in $M$.

Note that by contrast with Theorem~\ref{ICManifolds} which is for $3$-manifolds without boundary, Ikeda  \cite{Ik} has shown in the theorem below that for ``nice'' 3-manifolds with aspherical boundary, any graph with large enough genus which has a certain type of  automorphism of order $2$  has an achiral hyperbolic embedding in the double of the manifold. 

\begin{ikeda} \cite{Ik}  Let $M$ be a compact, connected, orientable, 3-manifold with non-empty aspherical boundary.  Then there is an integer $n_M$ such that for any abstract graph $\lambda$ with $\mathrm{genus}(\lambda)>n_M$ and no vertices of valence 1, any automorphism of order $2$ of $\lambda$ that does not restrict to an orientation preserving automorphism of a cycle in $\lambda$ can be induced by an orientation reversing involution of some hyperbolic embedding of $\lambda$ in the double of $M$.\end{ikeda}

Since the word ``graph'' is used in different ways by different authors, we note here that in the current paper as well as in \cite{fl4} a {\it graph} is defined to be a connected $1$-complex consisting of a finite number of vertices and edges such that every edge has two distinct vertices and there is at most one edge with a given pair of vertices. 

 In Section 2, we prove Theorem~\ref{theorem1}.  In Section 3, we prove Theorem~\ref{ICManifolds} making use of a proposition, which we then prove in Section 4.  Throughout the paper we work in the smooth category.  

\bigskip

\section{Achiral embeddings}

The goal of this section is to prove Theorem~\ref{theorem1}.  To that end, we prove the following proposition.
Note that we use $\mathrm{dim}_{\mathbb{Z}}(H_1(M,\mathbb{Z}))$ to denote the rank of the first $\mathbb{Z}$-homology group of $M$ and we use $\mathrm{dim}_{\mathbb{Z}_2}(H_1(M,\mathbb{Z}_2))$ to denote the rank of the first $\mathbb{Z}_2$-homology group of $M$.

\setcounter{prop}{0}
\begin{prop}  Let $S$ be a closed, orientable surface. 
Then for infinitely many closed, connected, orientable, irreducible 3-manifolds $Q$ such that $\mathrm{dim}_{\Z_2} (H_1 (Q,\Z_2)) = \mathrm{genus}(S)$, there is an embedding of $S$ in $Q$ which is pointwise fixed by an orientation reversing involution of $Q$.
\label{prop:inf}
\end{prop}

The proof of Proposition~\ref{prop:inf} will make use of the idea of a {\em disk-busting} curve, which is a simple closed curve in a handlebody that intersects every essential, properly embedded disk in the handlebody.   For example, a core of a solid torus is disk busting in the solid torus.  For a genus 2 handlebody whose fundamental group is generated by $a$ and $b$, an example of a disk-busting curve is one that includes into the fundamental group of the handlebody as $abab^{-1}$ (see Figure~\ref{diskbust}).  All handlebodies have disk-busting curves and Richard Stong \cite{S} gives an algorithm to recognize them.  For more on disk-busting curves see  \cite{H2} or  \cite{S}.

\begin{figure}[here]
\begin{center}
\includegraphics[width=.4\textwidth]{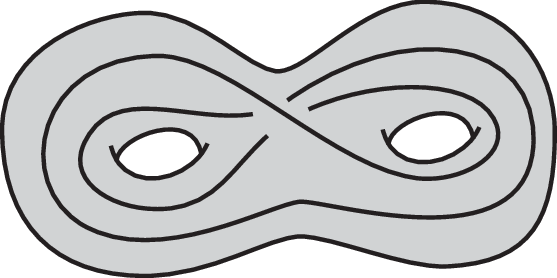}
\caption{A disk-busting curve in a genus 2 handlebody.}
\label{diskbust}
\end{center}
\end{figure}

\begin{proof} Let $g$ be the genus of a closed orientable surface $S$.  Let $M$ be the manifold obtained by gluing genus $g$ handlebodies $V_1$ and $V_2$ together along $S$ in such a way that there is an orientation reversing involution $h$ interchanging $V_1$ and $V_2$ which pointwise fixes the surface $S$.  Then $M$ is a closed, connected, orientable 3-manifold $M$ such that $\mathrm{dim}_{\Z_2} (H_1 (M,\Z_2)) = \mathrm{genus}(S)$.  However, $M$ is reducible.

In order to create an irreducible 3-manifold from $M$, we first remove neighborhoods $N_1$ and $N_2$ of identical disk busting curves $C_1$ and $C_2$ in the interiors of the handlebodies $V_1$ and $V_2$ such that $N_1$ and $N_2$ are interchanged by the involution $h$.
Note that since $ \mathrm{cl}(M-(S \cup  N_1 \cup  N_2))$ consists of two identical handlebodies from which neighborhoods of disk busting curves have been removed, the inclusion of $S$ into each component of $ \mathrm{cl}(M-(S \cup  N_1 \cup  N_2))$ is incompressible.

 We prove by contradiction that  $ \mathrm{cl}(V_i- N_i) \subset \mathrm{cl}(M-(S \cup  N_1 \cup  N_2))$ is irreducible.
Assume $ \mathrm{cl}(V_i- N_i)$ is reducible.  Let $F$ be a sphere in  $ \mathrm{cl}(V_i- N_i) $ that doesn't bound a ball. $F$ must separate  $S$ from $\partial N_i$ 
since every sphere in the handlebody $V_i$ bounds a ball in $V_i$.  Now take any compressing disk $D$ for $S = \partial V_i$ in the handlebody $V_i$ chosen so that it intersects $F$ minimally.  Since $F$ is a sphere, an innermost loop argument implies $D \cap F = \emptyset$.  This, however, implies that $D$ is disjoint from $\partial N_i$ and therefore that $S$ is compressible in $ \mathrm{cl}(V_i- N_i)$ which contradicts the conclusion above so $F$ must not exist and $ \mathrm{cl}(V_i- N_i)$ is irreducible.

Now we sew in identical non-trivial knot complements $Q_1$ and $Q_2$ 
to  $ \mathrm{cl}(M-(N_1 \cup  N_2))$
along $\partial N_1$ and $\partial N_2$ respectively so that the Seifert surfaces of the $Q_i$ are glued in where the meridians of the $N_i$ were.

Let $Q$ denote the 3-manifold obtained in this way.  Then the restriction $ h| \mathrm{cl}(M-( N_1 \cup  N_2))$ can be extended to an orientation reversing involution of $Q$ that pointwise fixes $S$.  Note that the components of $\mathrm{cl}(Q-(S \cup \partial Q_1 \cup  \partial Q_2))$
 that contain $S$ are homeomorphic to the corresponding components of $\mathrm{cl}(M-(S \cup  \partial N_1 \cup  \partial N_2))$ that contain $S$.
\medskip

\noindent{\bf Claim 1:}  The surfaces $S$, $\partial Q_1$, and $\partial Q_2$ are each incompressible in $Q$

\medskip

\noindent{\bf Proof of Claim 1:} Assume one of $S$, $\partial Q_1$, or $\partial Q_2$ is compressible in $Q$.  Then there is a compressing disk $D$ for one of these surfaces whose interior meets $S \cup \partial Q_1 \cup \partial Q_2$ transversally in a minimal number of components.  If the interior of $D$ intersects one of the surfaces $S$, $\partial Q_1$, or $\partial Q_2$, then an innermost loop on $D$ bounds a compressing disk $\Delta$ for that surface whose interior is disjoint from the other surfaces.  Now since each $Q_i$ is a knot complement, $\Delta$ cannot be contained in $Q_i$.  Thus $\Delta$ must be contained in some component of $\mathrm{cl}(Q-(S \cup  Q_1 \cup  Q_2))$.  Also, since the $C_i$ are disk busting for  the $V_i$, $\partial \Delta$ cannot be contained in $S$.  So $\Delta$ is a compressing disk for one of the $\partial N_i$ in $\mathrm{cl}(V_i-N_i)$.  By compressing along $\Delta$ we obtain a sphere $F$ in $\mathrm{cl}(V_i-N_i)$ separating $S$ from $\partial N_i$.  $F$ then is a reducing sphere for $\mathrm{cl}(V_i-N_i)$, but this contradicts the fact proven above that $\mathrm{cl}(V_i-N_i)$ is irreducible.
\bigskip

\noindent{\bf Claim 2:}  $Q$ is irreducible.  

\medskip

\noindent{\bf Proof of Claim 2:} Let $F$ be a sphere in $Q$ and assume without loss of generality that $F$ intersects $S \cup \partial Q_1 \cup \partial Q_2$ transversally in a minimal number of components.  If $F$ intersects any of $\partial Q_1$, $\partial Q_2$, or $S$, then there is an innermost loop on $F$ bounding a disk $D$ that is a compressing disk for $\partial Q_1$, $\partial Q_2$, or $S$ in  $ \mathrm{cl}(Q-(S \cup \partial Q_1 \cup \partial Q_2))$. But this violates Claim 1.  Thus $F$ is contained in some $Q_i$ or $\mathrm{cl}(V_i-Q_i)$.  However, since each $Q_i$ is a knot complement it is irreducible, and we saw above that $\mathrm{cl}(V_i-N_i)=\mathrm{cl}(V_i-Q_i)$ is irreducible.  Therefore $F$ bounds a ball in $Q$. 
\bigskip

Now, recall that $\mathrm{dim}_{\Z_2} (H_1 (M,\Z_2)) = \mathrm{genus}(S)$.  Also, it can be seen using a Mayer-Vietoris sequence that replacing the handlebodies $N_1$ and $N_2$ by the knot complements $Q_1$ and $Q_2$ does not change the first homology, since a meridional disk of $N_i$ is replaced by a Seifert surface in $Q_i$.  Thus $\mathrm{dim}_{\Z_2} (H_1 (Q,\Z_2)) = \mathrm{genus}(S)$, and hence $Q$ has the properties required by the proposition.

In order to prove that we can find infinitely many such 3-manifolds $Q$, we begin by letting $X_1=Q$.  Now it follows from Kneser-Haken finiteness, that since $X_1$ is compact and orientable, there is a finite constant $t_1$ such that $X_1$ cannot contain more than $t_1$ disjoint closed, non-parallel, incompressible surfaces (for example, see Proposition 1.7 in \cite{Ha} or see \cite{KH}).

Now let $n>t_1$, let $K_1$, \dots, $K_n$ be distinct non-trivial knots, and let $D$ denote a disk with $n$ holes.  Then $D\times S^1$ has $n+1$ torus boundary components which we denote by $S_1$, $S_2$, \dots, $S_n$, and $T_1$.  Now for each $j=1$, \dots $n$, we glue the knot complement $R_j=\mathrm{cl}(S^3-N(K_j))$ to the torus $S_j$ so that the boundary of a Seifert surface for $R_j$ is attached to $S_j\cap D$.  Then $Y_1=(D\times S^1)\cup R_1\dots \cup R_n$ is the complement of the connected sum $K_1 \# K_2 \# \dots \# K_n$, the $S_j$ are pairwise disjoint, non-parallel, incompressible tori in $Y_1$, and $\partial Y_1=T_1$.  Let $Y_1'$ denote a copy of $Y_1$.  Now we replace $Q_1$ by $Y_1$ and $Q_2$ by $Y_1'$ glued in such a way that the closed manifold $X_2$ obtained has an orientation reversing involution which interchanges $Y_1$ and $Y_1'$ pointwise fixing $S$.

 
Thus $X_2$ contains two copies of the tori $S_1$, $S_2$, \dots, $S_n$.  To see that these $2n$ tori are incompressible in $X_2$, suppose there is a compressing disk for one of the tori that intersects $\partial Y_1 \cup \partial Y_1'\cup  S$ transversally in a minimal number of components. An innermost loop on the disk would be a compressing disk for $\partial Y_1$, $\partial Y_1'$, or $S$, again contradicting Claim 1.  Thus the torus would have to be compressible in $Y_1$ or $Y_1'$.  But this is impossible by our construction of $Y_1$. Thus the $2n$ tori together with $\partial Y_1$ and $\partial Y_1'$ must all be incompressible in $X_2$.

It follows that $X_2$ contains at least $2n+2>t_1$ disjoint, non-parallel, incompressible tori, and thus $X_2$ is distinct from $X_1$.  By repeating this process, we can create an infinite sequence of such manifolds each containing more disjoint, non-parallel, incompressible tori than the previous manifold did.  \end{proof}

\bigskip

\begin{remark}  One of the referees observed that an alternate proof of Proposition 1 can be obtained using techniques of hyperbolic Dehn surgery together with results from \cite{QW}.  We prefer to maintain our original proof because our construction is explicit.\end{remark}
\bigskip

Since every graph embeds in a closed orientable surface, the following theorem is an immediate consequence of Proposition~\ref{prop:inf}.

\setcounter{theorem}{1}
\begin{theorem}  \label{theorem1} For every graph $\gamma$, there are infinitely many closed, connected, orientable, irreducible 3-manifolds $M$ such that some embedding of $\gamma$ in $M$ is pointwise fixed by an orientation reversing involution of $M$.

\end{theorem}

\bigskip

\section{Intrinsic Chirality}

The goal of this section is to prove our main result.  We begin with the following definition.

\begin{definition} We define the {\it genus of a closed surface} $S$ by $$\frac{2-\chi(S)}{2}=\frac{\mathrm{dim}_{\mathbb{Z}_2} (H_1 (S,\mathbb{Z}_2))}{2}.$$

Then the {\it genus of a graph} is defined as the minimum genus of any closed (orientable or non-orientable) surface in which the graph embeds.

\end{definition}

  It is worth pointing out that the genus of a non-orientable surface is not consistently defined in the literature.  Some papers use our definition, while others define the projective plane as having genus 1 instead of $\frac{1}{2}$.  
  \medskip

\begin{mainprop} Let $\gamma$ be a 3-connected graph with genus at least 2, and let $\Gamma$ be an embedding of $\gamma$  in a closed, connected, orientable, irreducible 3-manifold $M$ such that $(M,\Gamma)$ has an orientation reversing homeomorphism fixing every vertex of $\Gamma$.

Then there is an embedding $\Gamma'$ of $\gamma$ in a closed, connected, orientable 3-manifold $M'$ such that $(M',\Gamma')$ has an orientation reversing involution pointwise fixing $\Gamma'$ and $\mathrm{dim} _{\mathbb{Z}_2} (H_1 (M',\mathbb{Z}_2))\leq n_M$.
  \end{mainprop}
\medskip

The point of this proposition is that if we have an embedding $\Gamma$ of a graph $\gamma$ in a manifold $M$ such that $(M, \Gamma)$ has an orientation reversing {\it homeomorphism}, then we can find a (possibly different) manifold $M'$ and an embedding $\Gamma'$ of $\gamma$ in $M'$ such that $(M',\Gamma')$ has an orientation reversing {\it involution}.  Furthermore, even though $M'$ might be homologically more complicated than $M$, there is a bound on $\mathrm{dim} _{\mathbb{Z}_2} (H_1 (M',\mathbb{Z}_2))$ which depends only on $M$ and not on the graph $\gamma$ or the particular embedding of $\gamma$ in $M$.  

Note that the orientation reversing homeomorphism in the hypothesis of Proposition~\ref{mainprop}, leaves each edge  of $\Gamma$ setwise invariant since, according to our definition of a graph, there is at most one edge between any pair of vertices.  While this homeomorphism does not necessarily pointwise fix $\Gamma$, it is isotopic to a homeomorphism which does fix $\Gamma$ pointwise.

Proposition~\ref{mainprop} will be proved in the next section.  We now prove Theorem~\ref{ICManifolds} (restated below) by making use of Proposition~\ref{mainprop} together with the following result of Kobayashi.

\begin{kobayashi}\cite{Ko}  Let $X$ be a closed, orientable, 3-manifold admitting an orientation reversing involution $h$.  Then
 $$ \mathrm{dim}_{\mathbb{Z}_2} (H_1 (\mathrm{fix}(h),\mathbb{Z}_2)) \leq \mathrm{dim} _{\mathbb{Z}_2} (H_1 (X,\mathbb{Z}_2)) + \mathrm{dim} _{\mathbb{Z}} (H_1 (X,\mathbb{Z})).$$
  \end{kobayashi}
  \bigskip
 
\setcounter{theorem}{2}
\begin{theorem}   For every closed, connected, orientable, irreducible 3-manifold $M$, there is an integer $ n_M$ such that any abstract graph with no order $2$ automorphism which has a 3-connected minor $\lambda$ with $\mathrm{genus}(\lambda)>n_M$ is intrinsically chiral in $M$.\end{theorem}

\begin{proof}  Let $\gamma$ be a graph with no order $2$ automorphism.  Suppose for the sake of contradiction that $\gamma$ has an achiral embedding $\Gamma$ in the manifold $M$.  Let $n_M$ and $N_M$ be the constants associated with $M$ that were defined in Section~1.  Let $\lambda$ be a 3-connected minor of $\gamma$.  We will now show that $\lambda$ satisfies the inequality 

$$\mathrm{genus}(\lambda)\leq \mathrm{dim} _{\mathbb{Z}_2} (H_1 (M,\mathbb{Z}_2))+N_M=n_M.$$

\smallskip
 First observe that if $\mathrm{genus}(\lambda)<2$, then the above inequality is immediate since $N_M\geq 2$.  Thus we assume that $\mathrm{genus}(\lambda)\geq 2$.

Since $\Gamma$ is an achiral embedding of $\gamma$ in $M$, there is an orientation reversing homeomorphism $f$ of the pair $(M,\Gamma)$.  Let $\varphi$ denote the automorphism that $f$ induces on the graph $\Gamma$.  Note that even though $f$ does not necessarily have finite order, $\varphi$ has finite order because $\Gamma$ has a finite number of vertices.  Hence we can express the order of $\varphi$ as $2^rq$, where $r\geq 0$ and $q$ is odd.  Now since $f$ is orientation reversing and $q$ is odd, $g=f^{q}$ is an orientation reversing homeomorphism of $(M,\Gamma)$.  Also,  $g$ induces the automorphism $\varphi^q$ on $\Gamma$ and $\mathrm{order}(\varphi^q)=2^r$.  In particular, $g^{2^r}$ fixes every vertex of $\Gamma$.

If $r\geq 1$, then $g^{2^{r-1}}$ would induce an order two automorphism on $\Gamma$.  As we assumed that no such automorphism exists, we must have $r=0$.  Thus $g=g^{2^r}$ is an orientation reversing homeomorphism of $(M,\Gamma)$ which fixes every vertex of $\Gamma$.  

Since $\lambda$ is a  minor of the abstract graph $\gamma$, by deleting and/or contracting some edges of the embedding $\Gamma$ of $\gamma$ in $M$, we obtain an embedding $\Lambda$ of $\lambda$ in $M$.  Furthermore, by composing the homeomorphism $g$ with an isotopy in a neighborhood of each edge that was contracted, we obtain an orientation reversing homeomorphism of $(M,\Lambda)$ which fixes every vertex of $\Lambda$. 

Since $\lambda$ is a 3-connected graph with $\mathrm{genus}(\lambda)\geq 2$, we can now apply Proposition~\ref{mainprop} to get an embedding $\Lambda'$ of $\lambda$ in a closed, connected, orientable, 3-manifold $M'$ such that $(M',\Lambda')$ has an orientation reversing involution $h$ pointwise fixing $\Lambda'$ and

$$\mathrm{dim} _{\mathbb{Z}_2} (H_1 (M',\mathbb{Z}_2))\leq n_M=\mathrm{dim} _{\mathbb{Z}_2} (H_1 (M,\mathbb{Z}_2))+N_M.$$

\smallskip

Let $F$ be the component of the fixed point set $\mathrm{fix}(h)$ containing $\Lambda'$, and let $x\in\Lambda'$.
Since $h$ is a smooth involution,
we can now pick a neighborhood $N(x)$ which is homeomorphic to a ball and is invariant under $h$ (see for example, Theorem 2.2 in \cite{Bred}). Then by Smith Theory \cite{Sm}, since $h|N(x)$ is an orientation reversing involution of the ball $N(x)$, the fixed point set of $h|N(x)$ is either a single point or a properly embedded disk.  Since $N(x)\cap \Lambda$ contains more than one point, $\mathrm{fix}(h|N(x))$ is a properly embedded disk, and hence $F$ is a closed surface.  Thus

$$\mathrm{genus}(\lambda)\leq \mathrm{genus}(F)=\frac{2-\chi(F)}{2}$$

$$=\frac{\mathrm{dim}_{\mathbb{Z}_2}(H_1(F, \mathbb{Z}_2))}{2}\leq \frac{\mathrm{dim}_{\mathbb{Z}_2}(H_1(\mathrm{fix}(h), \mathbb{Z}_2))}{2}.$$

\smallskip

\noindent Hence we have the inequality $$2\mathrm{genus}(\lambda)\leq \mathrm{dim}_{\mathbb{Z}_2}(H_1(\mathrm{fix}(h), \mathbb{Z}_2)).$$

Also, since $h$ is an orientation reversing involution and $M'$ is a closed orientable manifold, we can apply Kobayashi's Theorem \cite{Ko} to obtain the inequality $$ \mathrm{dim}_{\mathbb{Z}_2} (H_1 (\mathrm{fix}(h),\mathbb{Z}_2)) \leq \mathrm{dim} _{\mathbb{Z}_2} (H_1 (M',\mathbb{Z}_2)) + \mathrm{dim} _{\mathbb{Z}} (H_1 (M',\mathbb{Z})).$$

\noindent It follows that $$ \mathrm{dim}_{\mathbb{Z}_2} (H_1 (\mathrm{fix}(h),\mathbb{Z}_2))\leq 2\mathrm{dim} _{\mathbb{Z}_2} (H_1 (M',\mathbb{Z}_2)).$$  

\smallskip

\noindent Combining the above inequalities, we now have $$\mathrm{genus}(\lambda)\leq \mathrm{dim} _{\mathbb{Z}_2} (H_1 (M',\mathbb{Z}_2)).$$  

\noindent But $M'$ was given by Proposition~\ref{mainprop} such that $$\mathrm{dim} _{\mathbb{Z}_2} (H_1 (M',\mathbb{Z}_2))\leq \mathrm{dim} _{\mathbb{Z}_2} (H_1 (M,\mathbb{Z}_2))+N_M.$$

\noindent Hence we obtain the required inequality $$\mathrm{genus}(\lambda)\leq \mathrm{dim} _{\mathbb{Z}_2} (H_1 (M,\mathbb{Z}_2))+N_M=n_M.$$ 

\noindent It now follows that if $\gamma$ has a 3-connected minor whose genus is greater than $n_M$, then $\gamma$ must be intrinsically chiral in $M$.  \end{proof}

\bigskip







\section{Proof of Proposition~\ref{mainprop}}

In the course of the proof of Proposition~\ref{mainprop}, we will use the well known  ``Half Lives, Half Dies" Theorem, which we state below. 
 See \cite{H1} or \cite{Ha} for a proof of this theorem. 

\begin{theorem} {\em (Half Lives, Half Dies)}  Let $M$ be a compact orientable $3$-manifold.  Then the following equation holds with any field coefficients
\[ \dim \left( Kernel (H_1(\partial M) \to H_1(M)) \right) = \frac{1}{2} \dim H_1(\partial M). \]
\label{thm:hlhdhom}
\end{theorem}

\begin{corollary}
Let $M$ be a manifold which has a torus boundary component $T$.  Then for any pair of generators $a$ and $b$ of $H_1(T,\mathbb{Z}_2)$, at least one of $a$ and $b$ is non-trivial in $H_1(M,\mathbb{Z}_2)$.
\label{cor:hlhd}
\end{corollary}

\begin{proof}
Suppose for the sake of contradiction that the generators $a$ and $b$ are both trivial in $H_1(M,\mathbb{Z}_2)$.  Attach handlebodies to all boundary components of $M$ except $T$ to form a new manifold $J$ with a single boundary component.  Then $a$ and $b$ are both trivial in $H_1(J,\mathbb{Z}_2)$.   Since $a$ and $b$ generate the homology of the only boundary component of $J$, we see that $\dim \left( Kernel (H_1(\partial J,\mathbb{Z}_2)) \to H_1(J,\mathbb{Z}_2) \right) =  \dim_{\Z_2} (H_1(\partial J,\mathbb{Z}_2)) = 2$.  But this contradicts the Half Lives, Half Dies Theorem.  Thus at least one of the generators of $H_1(T,\mathbb{Z}_2)$ must have been non-trivial in $M$.\end{proof}

\medskip

Note that it is tempting to assume that the Half Lives, Half Dies Theorem implies that one of $a$ or $b$ must be trivial in $J$.  But this is not always true. In particular, let $M$ denote the product of a torus and an interval.  Then no non-trivial curve in the boundary of $M$ is in the kernel. 
\medskip 

The proof of Proposition~\ref{mainprop} will make use of the Characteristic Decomposition Theorem for Pared Manifolds.  For readers who may not be familiar with pared manifolds, we include the following definitions and results.

\begin{definition}  Let $M$ be an orientable 
$3$-manifold and let $P$ be a set of 
disjoint incompressible annuli and tori in 
$\partial M$. Then we say $(M,P)$ is a \emph{pared 
$3$-manifold} \end{definition}        

\medskip\nobreak 

Note that a pared manifold is a special type of $3$-manifold 
with boundary pattern in the sense of Johannson 
\cite{Jo} as well as a $3$-manifold pair in the sense of 
Jaco-Shalen \cite{JS}.  The following definitions 
are consistent with those of \cite{Jo} and \cite{JS}.

\medskip\nobreak 

\begin{definition}      A pared manifold 
$(M,P)$ is \emph{simple} if it 
satisfies the following conditions: 
     
\par \noindent 1)  $M$ is irreducible and 
$\partial M-P$ is incompressible.

\par \noindent 2)  Every incompressible torus in 
$M$ is parallel to a torus component of $P$.

\par \noindent 3)  Every annulus $A$ in $M$ such that 
$\partial A\subseteq \partial M-P$ is 
either compressible or parallel to an annulus 
$A'$ in $\partial M$ with $\partial A'=\partial 
A$ such that $A'\cap P$ consists of zero or 
one annular component of $P$.\end{definition}     

\medskip\nobreak 
\begin{definition}   A pared manifold 
$(M,P)$ is \emph{Seifert fibered} if 
there is a Seifert fibration of $M$ for which $P$ 
is a union of fibers.  A pared manifold $(M,P)$ 
is \emph{ I-fibered} if there is an 
$I$-bundle map of $M$ over a surface $B$ such 
that $P$ is the preimage of $\partial 
B$.\end{definition}     

\medskip
\begin{paredJSJ}   \cite{JS}, \cite{Jo}  Let $(X,P)$ be a pared manifold where 
$X$ is irreducible and $\partial X-P$ is
incompressible.  Then, up to an isotopy of 
$(X,P)$, there is a unique set $\Omega $ of 
disjoint incompressible tori and annuli with 
$\partial \Omega\subseteq \partial X-P$ 
such that the following conditions hold:
\begin{enumerate}
\item  If $W$ is the closure of a 
component of $X-\Omega $, then the pared manifold 
$(W,W\cap (P\cup \Omega ))$ is either simple, 
Seifert fibered, or $I$-fibered. 
\medskip

\item  There is no set $\Omega '$ 
with fewer elements than $\Omega $ which 
satisfies the above.\end{enumerate}\end{paredJSJ}

\medskip\nobreak 
 \begin{paredThurston}   \cite{Th}  Suppose that a pared manifold $(M,P)$ is simple, $M$ 
is connected, and $\partial M$ is non-empty.  Then 
either $M-P$ admits a finite volume complete 
hyperbolic metric with totally geodesic boundary, $(M,P)$ is Seifert fibered, or $(M,P)$ is
$I$-fibered.\end{paredThurston}

\bigskip

We are now ready to prove Proposition~\ref{mainprop}.  Recall that the definition of the constant $n_M$ is given in Section 1.

\begin{prop}\label{mainprop}  Let $\gamma$ be a 3-connected graph with genus at least 2, and let $\Gamma$ be an embedding of $\gamma$  in a closed, connected, orientable, irreducible 3-manifold $M$ such that $(M,\Gamma)$ has an orientation reversing homeomorphism $g$ fixing every vertex of $\Gamma$.

Then there is an embedding $\Gamma'$ of $\gamma$ in a closed, connected, orientable 3-manifold $M'$ such that $(M',\Gamma')$ has an orientation reversing involution pointwise fixing $\Gamma'$ and $\mathrm{dim} _{\mathbb{Z}_2} (H_1 (M',\mathbb{Z}_2))\leq n_M.$ 
\end{prop}

\medskip

\begin{proof} Let $\Lambda$ denote either $\Gamma$ or $\Gamma$ with one edge deleted, and suppose that $\Lambda$ is contained in a ball $B$ in $M$.  Since $g(B)$ is isotopic to $B$ in $M$, we can assume that $g$ leaves $B$ setwise invariant.  Also, since $g$ fixes every vertex of $\Lambda$ and $\Lambda$ has at most one edge between any pair of vertices, $g$ takes each edge to itself.  Thus $g|B$ is isotopic to a homeomorphism which pointwise fixes $\Lambda$.  So we assume that $g$ leaves $B$ setwise invariant and pointwise fixes $\Lambda$.   Let $f$ be an embedding of $(B, \Lambda)$ in $S^3$.  Then $f\circ g\circ f^{-1}$ is an orientation reversing homeomorphism of $f(B)$ pointwise fixing $f(\Lambda)$.   Now, $f\circ g\circ f^{-1}$ can be extended to an orientation reversing homeomorphism of $S^3$ pointwise fixing $f(\Lambda)$.  

However, Jiang and Wang \cite{JW} showed that no graph containing $K_{3,3}$ or $K_5$ has an embedding in $S^3$ which is pointwise fixed by an orientation reversing homeomorphism of $S^3$.  Thus $\Lambda$ cannot contain $K_{3,3}$ or $K_5$, and hence is abstractly planar.  But this implies that $\mathrm{genus}(\Gamma)\leq 1$, which is contrary to our hypothesis.  Thus neither $\Gamma$ nor $\Gamma$ with an edge deleted can be contained in a ball in $M$.  We will use this result later in the proof.

\medskip

 Since the remainder of the proof is quite lengthy, we break it into steps.

\bigskip

\noindent{\bf Step 1:  We define a neighborhood $N(\Gamma)$. }
\bigskip

 Let  $V$ and $E$ be the sets of vertices and edges of $\Gamma$ respectively. For each vertex $v\in V$, define $N(v)$ to be a ball around $v$ in $M$ (i.e., a 0-handle containing $v$), and let $N(V)$ denote the union of the balls around the vertices.  Also, for each edge $e\in E$, let $N(e)=D \times I$ be a solid tube around $\mathrm{cl}(e-N(V))$ in $M$ (i.e., a 1-handle containing the portion of $e$ outside of the 0-handles). Then the intersection $N(V)  \cap  N(e)$ follows the standard convention for attaching 1-handles to 0-handles.  In other words, in $N(e)$ this intersection consists of the disks $D \times \{0\}$ and $D \times \{1\}$ in $\partial N(V)$.  Let $N(E)$ denote the union of the tubes around the edges.  Then $N(\Gamma)=N(E)\cup N(V)$ is a neighborhood of $\Gamma$.  
 
 For convenience we introduce the following terminology.  For each vertex $v$, we let $\partial' N(v)$ denote the sphere with holes $\partial N(v)\cap \partial N(\Gamma)$, and for each edge $e$ we let $\partial' N(e)$ denote the annulus $\partial N(e)\cap \partial N(\Gamma)$.  Thus $\partial N(\Gamma)=\partial' N(E)\cup\partial 'N(V)$.

Now since $g(\Gamma)=\Gamma$ fixing each vertex of $\Gamma$, we know that $g(N(\Gamma))$ is isotopic to $N(\Gamma)$ setwise fixing $\Gamma$ and fixing each vertex.  Thus we can modify $g$ by an isotopy (and by an abuse of notation, still refer to the map as $g$) so that $g(\Gamma)=\Gamma$, and  for each vertex $v$ and edge $e$ we have $g(N(v))= N(v)$ and $g(N(e))= N(e)$.  Because this modification was by an isotopy, our new
$g$ is still orientation reversing.  

 \bigskip

\noindent{\bf Step 2:  We split $\mathrm{cl}(M - N(\Gamma))$ along a family $\tau$ of JSJ tori and choose an invariant component $X$.}

\bigskip

Since $M$ is irreducible and we have assumed that $\Gamma$ is not contained in a ball, $\mathrm{cl}(M - N(\Gamma))$ is irreducible.  Thus we can apply the Characteristic Decomposition Theorem of Jaco-Shalen \cite{JS} and Johannson \cite{Jo} to get a minimal  family of incompressible tori $\tau$ for $\mathrm{cl}(M - N(\Gamma))$ such that each closed up component of $M - (N(\Gamma) \cup \tau)$ is either Seifert fibered or atoroidal.  Since the characteristic family $\tau$ is unique up to isotopy, we can again modify $g$ by an isotopy (and again by an abuse of notation still refer to the map as $g$) so that $g(\tau)=\tau$ and still have $g(\Gamma)=\Gamma$, $g(N(v))= N(v)$ and $g(N(e))= N(e)$ for each vertex $v$ and edge $e$.  Let $X$ be the closed up component of $M - (N(\Gamma) \cup \tau)$ containing $\partial N(\Gamma)$ (see for example Figure~\ref{JSJ}), then $g(X)=X$. 

\begin{figure}[here]
\begin{center}
\includegraphics[width=.75\textwidth]{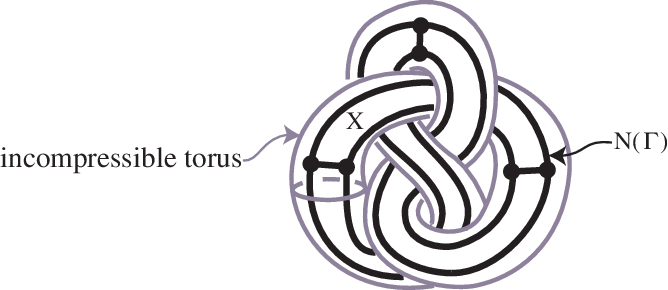}
\caption{$X$ is the closed up component of $M - (N(\Gamma) \cup \tau)$ between the grey incompressible torus and the black $\partial N(\Gamma)$.}
\label{JSJ}
\end{center}
\end{figure}

 Also, since $\Gamma$ is 3-connected, $\mathrm{genus}(\partial N(\Gamma))>1$.  Thus the component $X$ is not Seifert fibered, and hence is atoroidal.  Let $P $ denote the set of torus boundary components of $X$ together with the annuli that make up the components of $\partial 'N(E)$.  Since $\Gamma$ is 3-connected, $\partial X-P=\partial N(\Gamma) - \partial' N(E)=\partial 'N(V)$ is incompressible in $\mathrm{cl}(M - N(\Gamma))$.  It follows that $\partial X-P$ is incompressible in $X$.  Furthermore, $X$ is irreducible since $\mathrm{cl}(M - N(\Gamma))$ is irreducible and $X$ is a component of the JSJ decomposition of $\mathrm{cl}(M - N(\Gamma))$.  

\bigskip

In the next step, we consider the pared manifold $(X,P)$.

\noindent{\bf Step 3:  We show that any sphere obtained by capping off an annulus $A$ in  the JSJ decomposition of $(X,P)$ bounds a ball $B\subseteq M$ such that if the components of $\partial A$ are in distinct components of $\partial N(V)$, then $B$ meets $\Gamma-N(V)$ in a single edge, and if the components of $\partial A$ are in the same component of $\partial N(V)$, then $B$ is disjoint from $\Gamma-N(V)$.}
\bigskip

We now apply the Characteristic Decomposition Theorem for Pared Manifolds \cite{JS, Jo} to the pared manifold $(X,P)$. Since $X$ is atoroidal, this gives us a characteristic family $\sigma$ of incompressible annuli in $X$ with boundaries in $\partial X - P$ such that if $W$ is the closure of any component of $X-\sigma$, then the pared manifold $(W, W \cap (P \cup \sigma))$ is either simple, Seifert fibered, or $I$- fibered.  Once again, since the characteristic family $\sigma$  is unique up to isotopy, we can modify $g$ by an isotopy (and again by an abuse of notation, still refer to the map as $g$) so that $g(\sigma)=\sigma$. 

Let $A$ be an annulus component of $P\cup \sigma$, and let $S$ denote the sphere obtained by capping off $A$ by a pair of disjoint disks $D_1$ and $D_2$ in $ N(v_1)$ and $\partial N(v_2)$, where $v_1$ and $v_2$ may or may not be distinct vertices.  Suppose that each component of $M-S$ intersects more than one edge of $\Gamma-N(V)$.  Then by removing the vertices $v_1$ and $v_2$ and the edges that contain them we would obtain two non-empty subgraphs (see Figure~\ref{Annulus}).  But this contradicts our hypothesis that $\Gamma$ is 3-connected.  Thus one of the components of $M-S$ meets $\Gamma-N(V)$ in at most one edge of $\Gamma$.

\begin{figure}[here]
\begin{center}
\includegraphics[width=.5\textwidth]{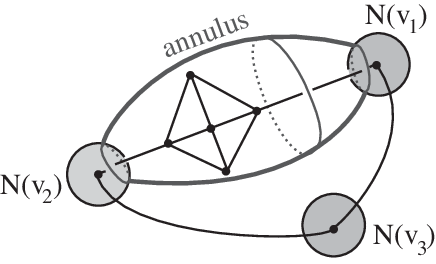}
\caption{There is more than one edge on each side of this capped off annulus, and hence $\Gamma$ is not $3$-connected.  }
\label{Annulus}
\end{center}
\end{figure}
\bigskip

Now, since $M$ is irreducible, one of the closed up components of $M-S$ is a ball $B$.  However, we showed at the beginning of our proof that neither $\Gamma$ nor $\Gamma$ with an edge removed can be contained in a ball in $M$.  Thus $B$ must be the closed up component of $M-S$ intersecting at most one edge of $\Gamma-N(V)$.  Furthermore, since the annulus $A$ is incompressible in $X$, if $v_1\not =v_2$ then there is some edge $e$ with vertices $v_1$ and $v_2$ such that $B$ contains $\mathrm{cl}(e-(N(v_1)\cup N(v_2)))$. On the other hand, suppose that $v_1=v_2$.  Then by our definition of a graph, every edge in  $\Gamma$ has two distinct vertices.  Thus the component of $M-S$ which meets $\Gamma-N(V)$ in at most one edge, is actually disjoint from $\Gamma-N(V)$.  In this case, $B$ must be disjoint from $\Gamma-N(V)$ as illustrated in Figure~\ref{v1v2}.  

\begin{figure}[here]
\begin{center}
\includegraphics[width=.5\textwidth]{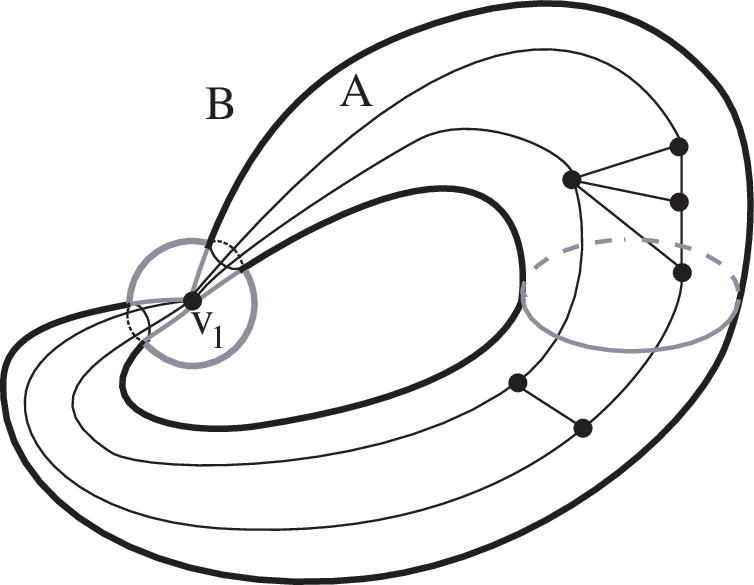}
\caption{The component $B$ of $M-S$ is disjoint from $\Gamma-N(V)$.  }
\label{v1v2}
\end{center}
\end{figure}

Note that since a ball cannot contain an incompressible torus, no torus boundary component of $X$ can occur inside of $B$.  It follows that every torus boundary component of $ X$ must also be a boundary component of $X-B$.

\bigskip

\noindent{\bf Step 4:  We define a collection of balls $U_{e_1}$, \dots, $U_{e_n}$, $V_{F_1}$, \dots, $V_{F_m}$ in $M$ such that the manifold $W=\mathrm{cl}(X-(U_{e_1}\cup\dots\cup U_{e_n}\cup V_{F_1}\cup \dots\cup V_{F_m}))$ is the closure of a single component of $X-(\sigma\cup P)$.}
\bigskip

Let $A$ be an annulus in $P\cup \sigma$ with one boundary in $\partial N(v_1)$ and the other boundary in $\partial N(v_2)$ where $v_1\not =v_2$.  By capping off $A$ with disks in $\partial N(v_1)$ and $\partial N(v_2)$ we obtain a sphere which (as we saw in Step 3) bounds a ball $B$ that meets $\Gamma$ in $\mathrm{cl}(e-(N(v_1)\cup N(v_2)))$ for some edge $e$.  

Note that by our definition of a graph, $e$ is the only edge of $\Gamma$ whose vertices are $v_1$ and $v_2$.  Thus we can let $\mathcal{C}_e$ denote the collection of all annuli in $P\cup \sigma$ with one boundary in $\partial N(v_1)$ and the other boundary in $\partial N(v_2)$.  Now we cap off the annuli in $\mathcal{C}_e$ with disks in $\partial N(v_1)$ and $\partial N(v_2)$ to obtain a collection of spheres which bound nested balls containing $\mathrm{cl}(e-(N(v_1)\cup N(v_2)))$.  Let $A_e$ denote the annulus in $\mathcal{C}_e$ which, when capped off in this way, is outermost with respect to the nested balls.  

Observe that the boundaries of $A_{e}$ bound disks $D_1\subseteq \partial N(v_1)$ and $D_2\subseteq \partial N(v_2)$ which each meet $\Gamma$ in a single point of $e$, and the sphere $A_{e}\cup D_1\cup D_2$ bounds a ball $U_{e}$ which contains $\mathrm{cl}(e-(N(v_1)\cup N(v_2)))$ and all of the other annuli in $\mathcal{C}_{e}$.  Furthermore, as we saw at the end of Step 3, no torus boundary component of $X$ can be contained in $U_e$.

We see as follows that $X-U_e$ has a single component.  If $A_e=\partial'N(e)$, then $U_e=N(e)$ and hence $\mathrm{cl}(X-U_e)=X$.  Thus we assume that $A_e\not =\partial'N(e)$.  In this case, $U_e \cap X$ is obtained by removing the interior of $N(e)$ from the ball $U_e$.  Thus $\partial(X-U_e)\cap U_e=A_e$ is an annulus.  Now let $p$ and $q$ be points in $X-U_e$.  Since $ X$ is path connected, there are paths $P$ and $Q$ in $X-U_e$ from $p$  and $q$ respectively to $\partial(X-U_e)\cap U_e$.  Now since $\partial(X-U_e)\cap U_e$ is an annulus, there is a path in $\partial(X-U_e)\cap U_e$ joining $P$ and $Q$.  Thus there is a path from $p$ to $q$ in $X-U_e$.  We will use this at the end of this step.


 
We repeat the above construction for each annulus in $P\cup \sigma$ with boundaries in distinct components of $\partial N(V)$.  This gives us a collection of pairwise disjoint balls $U_{e_1}$, \dots, $U_{e_n}$. Observe that for every edge $e_j$, the annulus $\partial'N(e_j)$ is in $P$ and its boundaries are in distinct components of $\partial N(V)$.  Thus every edge $e_j$ in $\Gamma$ is contained in some $U_{e_j}$.
It follows that $U_{e_1}\cup\dots \cup U_{e_n}$ contains $\mathrm{cl}(\Gamma-N(V))$ and contains every annulus in $P\cup \sigma$ with boundaries in distinct components of $\partial N(V)$.  

Next we consider an annulus $F$ in $\sigma$ both of whose boundaries are in a single $\partial N(v)$.  We saw in Step 3 that if we cap off $F$ by a pair of disjoint disks in $N(v)$, we obtain a sphere which bounds a ball that is disjoint from $\Gamma-N(v)$.  Thus the components of $\partial F$ are parallel in $\partial' N(v)$.  We now cap off every such annulus to obtain a collection of spheres which bound balls that are either nested or disjoint. 
Thus we can choose an annulus $F\in\sigma$ which, when capped off with disks in $N(v)$, bounds an outermost ball $V_F$ with respect to this nesting (see Figure \ref{Fi}).  

\begin{figure}[here]
\begin{center}
\includegraphics[width=.3\textwidth]{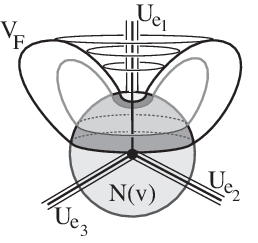}
\caption{$V_F$ is an outermost ball.}
\label{Fi}
\end{center}
\end{figure}

As we saw above for $U_e$, the ball $V_F$ cannot contain any torus boundary components of $X$.  In this case, $V_F\cap X$ is obtained by removing the interior of $V_F\cap N(v)$ from the ball $V_F$.   Thus $\partial(X-V_F)\cap V_F=F$ is an annulus.  Now by an argument analogous to the argument for $X-U_e$ we see that for any pair of points $p$ and $q$ in $X-V_F$ there is a path from $p$ to $q$ in $X-V_F$.  Thus $X-V_F$ also has a single component.




Now for each annulus $F\in \sigma$ with both components of $\partial F$ in a single $\partial N(v)$ such that $F$ is not contained in one of the balls $U_{e_j}$, we use the above argument to define a ball $V_{F}$.    In this way, we obtain a collection of pairwise disjoint balls $U_{e_1}$, \dots, $U_{e_n}$, $V_{F_1}$, \dots, $V_{F_m}$.  Furthermore, each of the sets $X-U_{e_j}$ and $X-V_{F_i}$ has a single component, and each $\partial U_{e_j}-\partial X$ and $\partial V_{F_i}-\partial X$ is an annulus in $\sigma\cup P$.  Now it follows that the manifold $$W=\mathrm{cl}(X-(U_{e_1}\cup\dots\cup U_{e_n}\cup V_{F_1}\cup \dots\cup V_{F_m}))$$ is the closure of a single component of $X-(\sigma\cup P)$ (see Figure~\ref{annulus}).


\begin{figure}[h]
\begin{center}
\includegraphics[width=.55\textwidth]{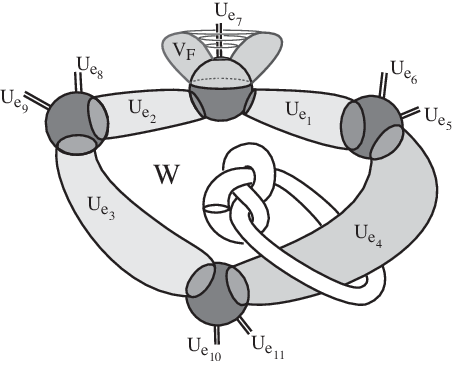}
\caption{$W$ is the closure of a single component of $X-\sigma$.}
\label{annulus}
\end{center}
\end{figure}

\bigskip

\noindent{\bf Step 5: We show that $g(W)=W$ and the pared manifold $(W, W \cap (P \cup \sigma))$ is simple.}
\bigskip

Recall that $g$ fixes each vertex and leaves each edge setwise invariant.  Also, $g(N(\Gamma))=N(\Gamma)$, $g(P)=P$, and $g(\sigma)=\sigma$.  Now, since the sets of balls $\{U_{e_1}, \dots, U_{e_n}\}$ and $\{V_{F_1}, \dots, V_{F_m}\}$ were chosen to be outermost, each of these sets is also invariant under $g$.  Furthermore, we know from Step 4 that each $U_{e_j}$ intersects $\Gamma- N(V)$ only in $e_j$.  Now since $g$ leaves each $e_j$ setwise invariant, $g$ must also leave each $U_{e_j}$ setwise invariant.  It follows that each $A_{e_j}$ must also be setwise invariant under $g$; and since each vertex is fixed by $g$, each boundary component of $A_{e_j}$ is also setwise invariant under $g$.  

Now let $v$ be a vertex such that some $F_i$ has both its boundary components in $\partial N(v)$.  Since $F_i$ is incompressible in $X$, the components of $\partial F_i$ bound disjoint disks in $\partial N(v)$ which each intersect at least one edge of $\Gamma$.  Since every vertex and edge of $\Gamma$ is invariant under $g$, it follows that $F_i$ and its boundary components are also setwise invariant under $g$.  Thus $V_{F_i}$ is setwise invariant under $g$.  Finally, since all of the $U_{e_j}$ and $V_{F_i}$ are setwise invariant under $g$, we know that $W$ must be setwise invariant under $g$ as well.

  To show that the pared manifold $(W, W \cap (P \cup \sigma))$ is simple, first recall that $W$ is the closure of a single component of $X-\sigma$.  Hence by JSJ for pared manifolds \cite{JS, Jo}, $(W, W \cap (P \cup \sigma))$ is either $I$-fibered, Seifert fibered, or simple as a pared manifold.  We see that  $(W, W \cap (P \cup \sigma))$ cannot be Seifert fibered or $I$-fibered as follows.

First observe that for every vertex $v$, there is some edge $e$ such that the ball $U_e$ meets $\partial N(v)$. It follows that $\partial W$ meets every component of $\partial N(V)$.  Furthermore, since every vertex $v$ has valence at least three, $\partial' N(v)$ is a sphere with at least three holes.  Also each $U_{e_j}$ contains at most one boundary component of $\partial' N(v)$, and each $V_{F_i}$ contains no boundary components of  $\partial' N(v)$.  Note that if $F_i$ has its boundaries in $\partial N(v)$, then $V_{F_i}$ separates $\partial'(N(v))$ into two components.

This means that for each vertex $v$, $W\cap \partial N(v)$ is obtained from $\partial'N(v)$ by deleting some (possibly zero) number of annuli.   Since $\partial' N(v)$ is a sphere with at least three holes, deleting a collection of annuli may create new components, but some component must still be a sphere with at least three holes.  It follows that the component of $\partial W$ meeting $\partial N(\Gamma)$ has genus more than one, and thus the pared manifold $(W, W \cap (P \cup \sigma))$ cannot be Seifert fibered.

Next, suppose for the sake of contradiction that the pared manifold $(W, W \cap (P \cup \sigma))$ is $I$-fibered.  By definition of $I$-fibered for pared manifolds, this means that there is an 
$I$-bundle map of $W$ over a base surface $Y$ such 
that $W \cap (P \cup \sigma)$ is in the pre-image of $\partial 
Y$.  It follows that $Y$ must be homeomorphic to a component of  $\partial 'N(V)\cap W$, and hence must be a sphere with holes.  In particular, the base surface $Y$ is orientable.

Now since $W$ is an orientable $3$-manifold which is $I$-fibered over an orientable surface, it follows that $W$ must actually be a product $Y\times I$.  Thus
$W\cap (P\cup \sigma )=\partial Y \times I$, and $Y_{0}=Y\times \lbrace 
0\rbrace $ and $Y_{1}=Y\times \lbrace 1\rbrace $ are the only components of $\partial' N(V)\cap W$.  However, since $\partial W$ meets every component of $\partial N(V)$, this means that $\Gamma$ contains at most two vertices.  But this contradicts our hypothesis that $\Gamma$ is 3-connected.  Therefore, the pared $(W, W \cap (P \cup \sigma))$ is not $I$-fibered.  Since it is also not Seifert fibered, it must be simple.   
 
 \bigskip
 
 \noindent{\bf Step 6: We prove that $g|W$ is isotopic to an orientation reversing involution $h$ of $(W, W\cap (P\cup \sigma ))$.}

\bigskip

 Now it follows from Thurston's Hyperbolization
Theorem for Pared Manifolds \cite{Th} applied to the simple pared manifold $(W, W\cap (P\cup
\sigma ))$ that $W-(W\cap (P\cup \sigma ))$ admits a finite volume
complete hyperbolic metric with totally geodesic boundary.  Let $D$
denote the double of $W-(W\cap (P\cup \sigma ))$ along its boundary.
Then $D$ is a finite volume hyperbolic manifold, and $g|W$ can be doubled to obtain an orientation reversing  homeomorphism of $D$ (which we still call $g$) taking each copy of $W-(W\cap (P\cup \sigma ))$ to itself.
Now by Mostow's Rigidity Theorem~\cite{Mo} applied to $D$, the homeomorphism $g:D\to D$ is homotopic to an orientation reversing finite order isometry $h:D\to D$
that restricts to an isometry of $W- (W\cap(P\cup \sigma ))$. By removing
horocyclic neighborhoods of the cusps of $W-(W\cap(P\cup \sigma ))$,
we obtain a copy of the pair $(W, W\cap (P\cup \sigma ))$ which is
contained in $W-(W\cap(P\cup \sigma))$ and is setwise invariant under $h$.  We abuse notation and now consider $h$ to
be an orientation reversing finite order isometry of $(W, W\cap (P\cup \sigma ))$ instead of a copy of $(W, W\cap (P\cup \sigma ))$.  Furthermore, $h$ induces isometries on the collection of tori and annuli in
$W\cap (P\cup \sigma )$ with respect to a flat metric.  Furthermore, the sets
$\partial' N(V)\cap W$, $\partial'N(E)\cap W$, and $\tau \cap W$ are
each setwise invariant under $h$.  Finally, it follows from
Waldhausen's Isotopy Theorem~\cite{Wa} that $h$ is actually
isotopic to $g|W$ by an isotopy leaving
$W\cap (P\cup \sigma )$ setwise invariant.

 Now, recall that the boundary components of $W$ consist of tori in $\tau$, and the union of spheres with holes in $\partial' N(V)\cap W$ together with annuli in $P\cup \sigma$. Recall from the first paragraph in Step 5 that $g$ setwise fixes each annulus $A_{e_j}\subseteq \partial U_{e_j}$ with boundaries in distinct components of $\partial N(V)$, each annulus $F_i\subseteq \partial V_{F_i}$ with both boundaries in a single component of $\partial N(V)$, each component of $\partial A_{e_j}$, and each component of $\partial F_i$.  Since $h$ is
isotopic to $g|W$ by an isotopy leaving
$W\cap (P\cup \sigma )$ setwise invariant, $h$ leaves invariant the same sets as $g$.  It follows that for each vertex $v$, we have $h(\partial' N(v))\cap W)=\partial' N(v)\cap W$, and $h$ takes
each component of $W\cap \partial N(v)$ to itself, leaving each boundary component setwise invariant.  

Since $h$ has finite order, $h$ restricts to a finite order homeomorphism of every component of $W\cap \partial N(V)$.  We saw in Step 5 that for every vertex $v$, at least one component $C_v$ of $W\cap \partial N(v)$ is a sphere with at least three holes.  Since $h$ restricts to a finite order homeomorphism of $C_v$ taking each boundary component of $ C_v$ to itself, $h$ must be a reflection of $C_v$ which also reflects each component of $\partial C_v$. Now $h^2$ is a finite order, orientation preserving isometry of $W$ that pointwise fixes the surface $C_v$.  It follows that $h^2$ is the identity, and hence $h$ is an involution of $W$.  

\bigskip
 
 \noindent{\bf Step 7: We extend $h$ to an orientation reversing involution of $X\cup N(\Gamma)$ which pointwise fixes an embedding $\Gamma'$ of $\gamma$.}

\bigskip

Observe that since every annulus in $P\cup \sigma$ is incompressible in $W$, no component of $W\cap \partial N(V)$ can be a disk.  Thus every component of $W\cap \partial N(V)$ is a sphere with two or more holes.  

As we saw in Step~6, for each vertex $v$, $h$ reflects some component $C_v$ of $W\cap \partial N(v)$ which is a sphere with at least three holes, and $h$ reflects every component of $\partial C_v$.  Let $b_0$ denote some boundary component of $C_v$.   Then $b_0$ is also a boundary component of either an annulus $A_{e_j}$ or an annulus $F_i$.  Since $h$ reflects $b_0$, we know that $h$ must also reflect the annulus $A_{e_j}$ or $F_i$, whichever contains $b_0$ in its boundary.  Since the boundaries of the annulus are not interchanged, $h$ must also reflect each boundary component of $A_{e_j}$ or $F_i$.  Below we extend $h$ to $U_{e_j}$ or $V_{F_i}$.

First we consider the case where $b_0$ is in the boundary of an annulus $A_{e_j}\subseteq \partial U_{e_j}$.   Let $D_j$ and $D'_j$ denote the disks whose union is in $\mathrm{cl}(\partial U_{e_j}-A_{e_j})$.  Then $D_j$ and $D_j'$ each meet $\Gamma$ in a single point of $e_j$. Since $h$ reflects the annulus $A_{e_j}$ together with each boundary component of $A_{e_j}$, we can extend $h$ radially to the disks $D_j$ and $D_j'$ to get a reflection of the sphere $A_{e_j}\cup D_j\cup D_j'$ pointwise fixing a circle containing the points $D_j\cap e_j$ and $D_j'\cap e_j$.  Recall that the sphere $A_{e_j}\cup D_j\cup D_j'$ bounds the ball $U_{e_j}$ in $M$.  Now, we  express $U_{e_j}$ as a product $D_j\times I$ whose core $\overline{e_j}$ has endpoints $D_j\cap e_j$ and $D_j'\cap e_j$  (see Figure~\ref{knotted}).  Then we extend $h$ from a reflection of the sphere $A_{e_j}\cup D_j\cup D_j'$ to a reflection of the product $D_j\times I$ which pointwise fixes the core $\overline{e_j}$.

\begin{figure}[here]
\begin{center}
\includegraphics[width=.55\textwidth]{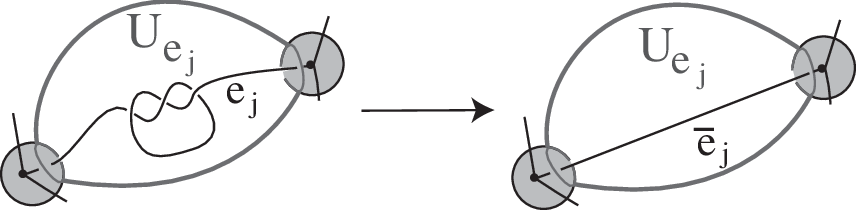}
\caption{We can think of $U_{e_j}$ as a product $D_j\times I$ with core $\overline{e_j}$.}
\label{knotted}
\end{center}
\end{figure}

  Next we consider the case where $b_0$ is a boundary component of an annulus $F_i\subseteq \partial V_{F_i}$ which has both boundaries in a single $\partial N(v)$.  Recall that $\mathrm{cl}(\partial V_{F_i}-F_i)$ consists of disks $D_i$ and $D_i'$ properly embedded in $N(v)$.  Without loss of generality, $D_i$ and $D_i'$ each meet $\Gamma$ at a single point on an edge.  Since $h$ reflects the annulus $F_i$ together with each of its boundary components, we can extend $h$ radially to the disks $D_i$ and $D_i'$ to get a reflection of the sphere $F_i\cup D_i\cup D_i'$ pointwise fixing a circle containing the points $\Gamma\cap D_i$ and $\Gamma\cap D_i'$.  Thus we can extend $h$ to a reflection of the ball $V_{F_i}$ which pointwise fixes a disk containing the segment $V_{F_i}\cap \Gamma$ (see Figure~\ref{Bi}).
    
\begin{figure}[here]
\begin{center}
\includegraphics[width=.4\textwidth]{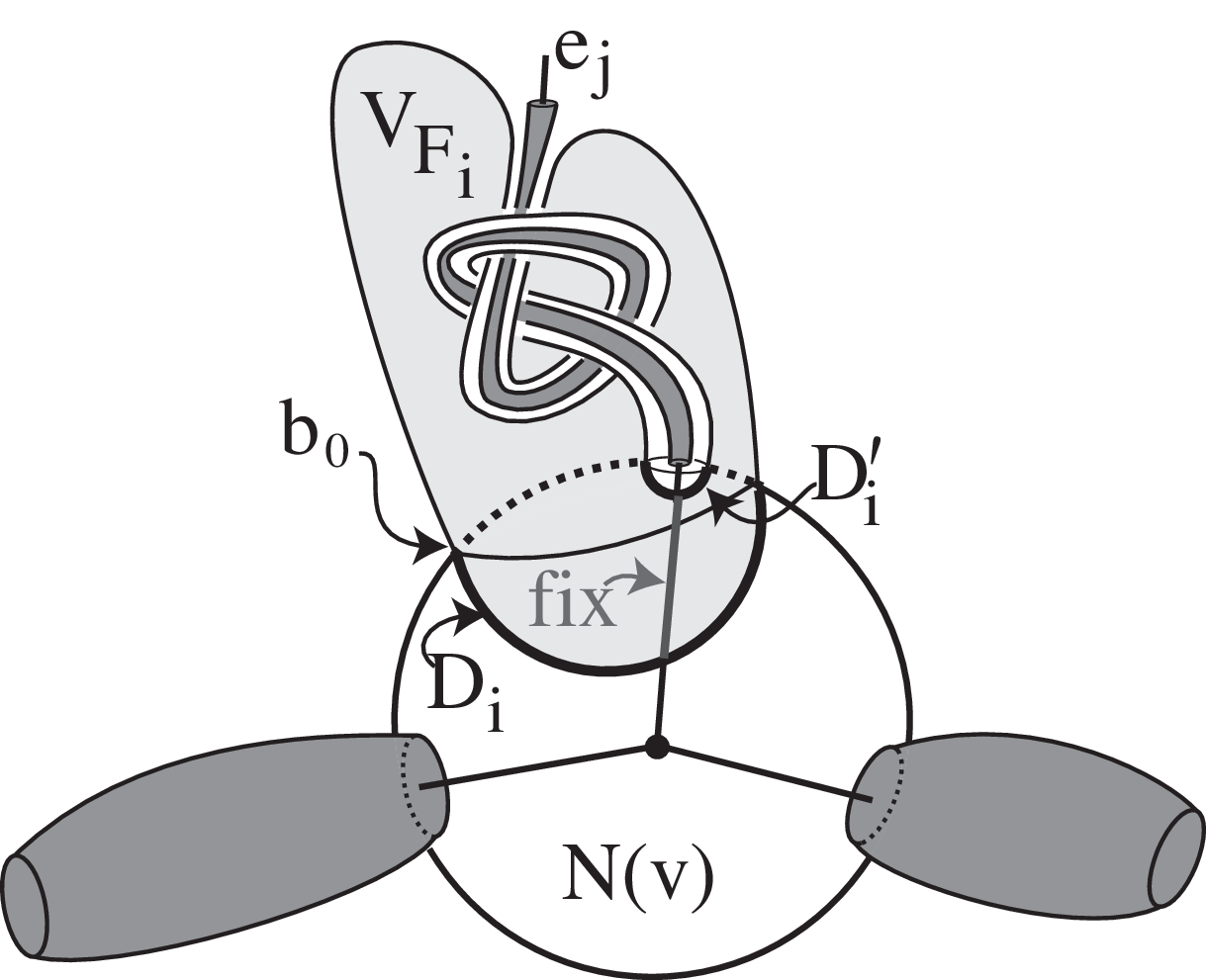}
\caption{We extend $h$ to a reflection of the ball $V_{F_i}$ which pointwise fixes $\Gamma\cap V_{F_i}$.  }
\label{Bi}
\end{center}
\end{figure}

In either case, the extension of $h$ reflects the sphere with holes $C_v$, and one of the balls $U_{e_j}$ or $V_{F_i}$ depending on whether $b_0$ is a boundary component of $A_{e_j}$ or $F_i$, respectively.  Next we let $S_1$ denote the union of $C_v$ together with the annulus $A_{e_j}$ or $F_i$ glued along $b_0$.  Now $h$ reflects $S_1$ taking every boundary component of $S_1$ to itself, and hence reflecting every boundary component of $S_1$.  Let $b_1$ be a boundary component of $S_1$. If $b_1$ is not the other boundary of the annulus $A_{e_j}$ or $F_i$, then we repeat the above argument with $b_1$ in place of $b_0$.

\begin{figure}[here]
\begin{center}
\includegraphics[width=.45\textwidth]{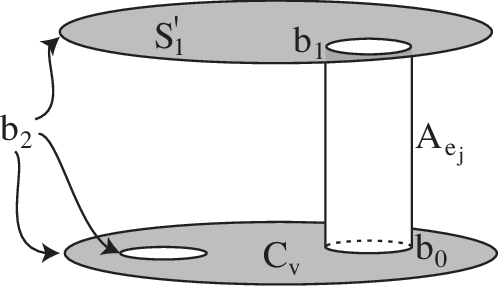}
\caption{In this illustration, we have three choices for the boundary component $b_2$ of $S_2=C_v\cup A_{e_j} \cup S_1'$.  }
\label{induction}
\end{center} 
\end{figure}

 If $b_1$ is the other boundary of the annulus $A_{e_j}$ or $F_i$, then $b_1$ is also a boundary of some other component $S_1'$ of $W\cap \partial N(V)$, as illustrated in Figure~\ref{induction}.  In this case, since $b_1$ is reflected by $h$ and every boundary component of $S_1'$ is setwise invariant under $h$, we know that $h$ must reflect $S_2=S_1\cup S_1'$.  Now let $b_2$ denote a boundary component of $S_2$, and repeat the above argument with $b_2$ in place of $b_0$.  

In general, for a given surface $S_n$ obtained in this way, the surface $S_{n+1}$ is the union of $S_n$ together with either an annulus of the form $A_{e_j}$ or $F_i$ or a sphere with at least two holes contained in $W\cap \partial N(V)$.  Furthermore, $S_{n+1}$ is reflected by $h$. This process will only stop when the surface that we obtain has no boundary components.  Since $\partial W$ has only one component which intersects $\partial N(V)$, the closed surface that we obtain in this way must be this component.  Thus we have extended $h$ to an orientation reversing involution of each of the balls $U_{e_1}$, \dots, $U_{e_n}$, $V_{F_1}$, \dots, $V_{F_m}$.

Since $N(E)\cup(X-W)\subseteq U_{e_1}\cup\dots U_{e_n}$, at this point we have defined $h$ on $X$, and the only part of $N(\Gamma)$ on which we have not defined $h$ is $N=N(V)-(V_{F_1}\cup\dots\cup V_{F_m})$.  Then $N$ is a collection of disjoint balls (for example in Figure \ref{Bi}, $N(v)-V_{F_i}$ is two balls one of which contains $v$).  Also, $h$ is a reflection of each component of $\partial N$ that fixes each point in $\partial N \cap \Gamma$. 
 Now we extend $h$ radially to a reflection of each ball of $N$ in such a way that $h$ pointwise fixes each component of $N \cap \Gamma$.  Thus $h$ is a reflection of each component of $N(V)$ which pointwise fixes $N(V)\cap \Gamma$.

We have now extended $h$ to an orientation reversing involution of the manifold

$$Y=W\cup V_{F_1}\cup\dots\cup V_{F_m}\cup U_{e_1}\cup\dots\cup U_{e_n}\cup N.$$  

Recall from the end of Step 4 that $\partial X$ and $\partial W$ have the same collection of tori in their boundary components.  Furthermore, we have filled in the boundary component of $W$ meeting $\partial N(\Gamma)$ with a collection of balls in $X\cup N(\Gamma)$.  Thus in fact $Y=X\cup N(\Gamma)$ and $\partial Y$ is the collection of tori in $\partial X$.

Finally, we define a new embedding $\Gamma'$ of $\gamma$ in $X\cup N(\Gamma)$ as follows. Let $\Gamma'\cap N(V)=\Gamma \cap N(V)$.  Then for each edge $e_j$ define an embedding of $e_j-N(V)$ in $\Gamma'$ as the core $\overline{e_j}$ of $U_{e_j}=D_j\times I$, which we know is pointwise fixed by $h$ according to the way we extended $h$ to $U_{e_j}$ (recall Figure~\ref{knotted}).  
\bigskip

 \noindent{\bf Step 8:  We prove that if an essential curve in a component of $\partial (X\cup N(\Gamma))$ compresses in $M$, then it compresses in $X\cup N(\Gamma)$.}
\bigskip

 Let $T_1$,\dots, $T_r$ be the components of $\partial (X\cup N(\Gamma))$.  Then each $T_i$ is contained in the characteristic family $\tau$, and hence is incompressible in $\mathrm{cl}(M-N(\Gamma))$.   
 
 Suppose that an essential curve $\lambda_i$ on some $T_i$ compresses in $M$.   Let $D_i$ be a compressing disk for $\lambda_i$ whose intersection with the set of tori $\{T_1,\dots, T_r\}$ is minimal.  Let $D=D_i$ if the interior of $D_i$ is disjoint from $T_i$.  Otherwise, there exists some $D$ in the interior of $D_i$ such that $D$ is a compressing disk for $T_i$ whose interior is disjoint from $T_i$.  In either case, the intersection of $D$ with $\{T_1,\dots, T_r\}$ is minimal.  
 
 Suppose that $D$ contains at least one curve of intersection in its interior.  Hence there is an innermost disk $\Delta$ on $D$ which is a compressing disk for some $T_j$ with $j\not=i$.  Since $T_j$ compresses in $M$ but is incompressible in $\mathrm{cl}(M-N(\Gamma))$, we know that $\Delta$ intersects $\Gamma$.

 Since $M$ is irreducible, any compressible torus is separating in $M$.  Thus we can let $X_j$ denote the closed up component of $M-T_j$ containing $X$ and let $V_j$ denote the closed up component of $M-T_j$ whose interior is disjoint from $X$.  Now let $S$ denote the region of $D$ which is adjacent to the innermost disk $\Delta$.  Then $S\subseteq V_j$, since $\Delta\subseteq X_j$.  Also, $\partial D\subseteq T_i\subseteq X\subseteq X_j$ implies that $S$ is adjacent to another region of $D$ which is contained in $X_j$.  In particular, there must be another circle of intersection $\alpha$ of $D\cap T_j$ which bounds a disk $\overline{D}\subseteq D$ such that $\overline{D}$ contains $\Delta\cup S$.   We illustrate the abstract disk $D$ and its intersections with $T_j$ in Figure~\ref{disk}.  The white regions in the figure are contained in $V_j$, and the grey regions are contained in $X_j$.  Note we do not illustrate any circles of intersection of $D$ with any $T_k$ with $k\not =j$.

\begin{figure}[here]
\begin{center}
\includegraphics[width=.3\textwidth]{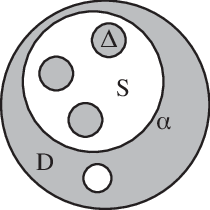}
\caption{A picture of the abstract disk $D$ and its circles of intersection with  $T_j$.  }
\label{disk}
\end{center}
\end{figure}

Now, since the intersection of $D$ with the tori $T_1$,\dots, $T_r$ is minimal, all of the curves of intersection of $D\cap T_j$ must be essential on $T_j$.  In particular, $\partial \Delta$ and $\alpha$ must both be essential on $T_j$.  Since there cannot be two essential, disjoint, non-parallel curves on a torus, this means that $\alpha$ is parallel to $\partial \Delta$ on $T_j$.  It follows that $\alpha$ must bound a disk $\overline{\Delta}$ which is parallel to $\Delta$ in $M$.  In particular, since the interior of $\Delta$ is disjoint from $T_1$,\dots, $T_r$, the interior of $\overline{\Delta}$ is as well.  But now by replacing the disk $\overline{D}$ with the disk $\overline{\Delta}$ in the compressing disk $D$ we obtain a new compressing disk $D'$ for $T_i$ which has fewer curves of intersection with $T_1$,\dots, $T_r$ than $D$ has.  From this contradiction we conclude that the interior of $D$ must be disjoint from $T_1\cup\dots\cup T_r$, and hence $D\subseteq X\cup N(\Gamma)$.

If $\partial D=\lambda_i$, then $\lambda_i$ compresses in $X\cup N(\Gamma)$ as required.  Otherwise, the compression disk $D$ was contained in the interior of the original disk $D_i$ and $\partial D\subseteq T_i$.  In this case, since the intersection of $D_i$ with the set of tori $\{T_1,\dots, T_r\}$ was minimal, $\partial D$ is essential in $T_i$.  But now $\partial D$ and $\lambda_i$ are disjoint essential curves on $T_i$.  Hence as we saw above, the disks $D$ and $D_i$ must be parallel in $M$.  Now, since $D\subseteq X\cup N(\Gamma)$, it follows that $\lambda_i$ must compress in $X\cup N(\Gamma)$ as well.

\bigskip

 \noindent{\bf Step 9: We fill each component $T_i$ of $\partial (X\cup N(\Gamma))$ with a solid torus such that $h$ extends to an involution of the resulting manifold $M'$, and $M'$ satisfies the condition below.}
\bigskip

\noindent {\bf Condition:} {\it  If $T_i$ is compressible in $M$, then every element of $H_1 (T_i,\mathbb{Z}_2)$ is trivial in $H_1 (M',\mathbb{Z}_2)$, and if $T_i$ is incompressible in $M$ then at least one non-trivial element of $H_1 (T_i,\mathbb{Z}_2)$ is trivial in $H_1 (M',\mathbb{Z}_2)$.}  
\bigskip

Let $T_i$ be a component of $\partial (X\cup N(\Gamma))$.  By Corollary~\ref{cor:hlhd} at the beginning of Section 4 there is a curve $\mu_i$ on $T_i$ which is non-trivial in $H_1( X\cup N(\Gamma),\mathbb{Z}_2)$.  Also, we know from Step 8 that if some essential curve $\lambda_i$ on $T_i$ compresses in $M$, then $\lambda_i$ also compresses in $X\cup N(\Gamma)$.  In particular, $\lambda_i$ is not homologous in $T_i$ to $\mu_i$.

Now suppose that for some $j\not =i$, the involution $h$ interchanges $T_i$ and $T_j$.  Since $h:X\cup N(\Gamma)\to X\cup N(\Gamma)$ is a homeomorphism and $\mu_i$ is a curve on $T_i$ which is non-trivial in $H_1( X\cup N(\Gamma),\mathbb{Z}_2)$, we know that $h(\mu_i)$ is a curve on $T_j$ which is also non-trivial in $H_1(X\cup N(\Gamma),\mathbb{Z}_2)$.  Now we fill $X\cup N(\Gamma)$ along $T_i$ by adding a solid torus $V_i$ with its meridian attached to the non-trivial curve $\mu_i$, and we fill along $T_j$ by adding a solid torus $V_j$ with its meridian attached to $h(\mu_i)$.  Then we extend the involution $h$ radially in $V_i\cup V_j$ (abusing notation and still calling the involution $h$).   We repeat this process for every boundary component of $X\cup N(\Gamma)$ which is not setwise invariant under $h$.   

Thus    for every $T_i$ along which we have glued a solid torus $V_i$, the curve $\mu_i$ on $T_i$ is now trivial in $H_1 (X\cup N(\Gamma)\cup V_i,\mathbb{Z}_2)$.  Furthermore, if $T_i$ is compressible in $M$, then there is an essential curve $\lambda_i$ on $T_i$ which compresses in $X\cup N(\Gamma)$.  Hence, $\lambda_i$ is not homologous to $\mu_i$ in $H_1 (T_i,\mathbb{Z}_2)$, and together they generate $H_1 (T_i,\mathbb{Z}_2)$.  Furthermore, both $\lambda_i$ and $\mu_i$ are trivial in $H_1 (X\cup N(\Gamma)\cup V_i,\mathbb{Z}_2)$.

Let $Z$ be the manifold that we have obtained by filling all of the boundary components of $X\cup N(\Gamma)$ which are not setwise fixed by $h$, and let $T_i$ be a component of $\partial Z$.  Recall from Step 7 that $\Gamma'$ is an embedding of $\gamma$ in $X\cup N(\Gamma)$ which is pointwise fixed by $h$.   Now $h:Z\to Z$ is an orientation reversing involution pointwise fixing $\Gamma'$, and $T_i$ is setwise invariant under $h$.  Furthermore, by the proof of Step~6, $h|T_i$ is an isometry with respect to a flat metric.  Since $h$ is an orientation reversing involution of $Z$ and $T_i$ is a boundary component of $Z$, $h|T_i$ is also an orientation reversing involution. Thus $h|T_i$ is either a reflection pointwise fixing two parallel circles on $T_i$, or a composition of a reflection and an order $2$ rotation of $T_i$. In either case, for any non-trivial element $a_i\in H_1( T_i,\mathbb{Z})$, there is a non-trivial curve $b_i\in H_1( T_i,\mathbb{Z})$ such that $a_i$ and $b_i$ meet transversely in a single point and $h(b_i)$ is homologous to $\pm b_i$ in $H_1( T_i,\mathbb{Z})$.


Now suppose that some essential curve $\lambda_i$ on $T_i$ compresses in $M$.  Then by Step 8, $\lambda_i$ also compresses in $X\cup N(\Gamma)$, and hence in $Z$.  Pick a non-trivial curve $b_i$ on $T_i$ that intersects $\lambda_i$ in a single point such that $h(b_i)$ is homologous to $\pm b_i$ in $H_1( T_i,\mathbb{Z})$.  Then $\langle \lambda_i, b_i\rangle=H_1( T_i,\mathbb{Z})$.  Also, since $\lambda_i$ is null homologous in $H_1( Z,\mathbb{Z}_2)$, by Corollary~\ref{cor:hlhd}, $b_i$ is non-trivial in $H_1( Z,\mathbb{Z}_2)$.  Now we fill $Z$ along $T_i$ by adding a solid torus $V_i$ with its meridian attached to the non-trivial curve $b_i$.  Since $h(b_i)$ is homologous to $ \pm b_i$ in $H_1( T_i,\mathbb{Z})$, we can extend $h$ radially to an orientation reversing involution of the solid torus $V_i$.  Then $h:Z\cup V_i\to Z\cup V_i$ is an orientation reversing involution, and both $\lambda_i$ and $b_i$ are trivial in $H_1( Z\cup V_i,\mathbb{Z}_2)$.  

On the other hand, suppose that some $T_i$ is incompressible in $M$.   By Corollary~\ref{cor:hlhd}, there is some curve $b_i$ on $T_i$ which is non-trivial in $H_1( Z,\mathbb{Z}_2)$, and $b_i$ is homologous to $\pm h(b_i)$ in $H_1( T_i,\mathbb{Z})$.  Now we fill $T_i$ by adding a solid torus $V_i$ with its meridian attached to the curve $b_i$, and then extend $h$ to $V_i$.  Then $h:Z\cup V_i\to Z\cup V_i$ is again an orientation reversing involution, and $b_i$ is trivial in $H_1( Z\cup V_i,\mathbb{Z}_2)$.

In this way, we glue a solid torus to each of the $T_i$ in $\partial(X\cup N(\Gamma))$ to obtain a closed manifold $M'$ satisfying the required condition.  Since $\Gamma'\subseteq X\cup N(\Gamma)$, this gives us an embedding $\Gamma'$ of $\gamma$ in $M'$.  Furthermore, we have extended $h$ to an orientation reversing involution of $(M', \Gamma')$ which pointwise fixes $\Gamma'$.  
\bigskip

To prove the proposition it remains to show that $\mathrm{dim} _{\Z_2} (H_1 (M',\Z_2))\leq n_M.$  We do this in steps 10 and 11.
 
\bigskip
 
 \noindent{\bf Step 10: We prove that at most $N_M$ of the tori in $\partial (X\cup N(\Gamma))$ are incompressible in $M$.}
\bigskip

Recall that $N_M$ is either the number of non-parallel disjoint incompressible tori in $M$ or $2$, whichever is larger.  If no pair of distinct tori of $\partial(X\cup N(\Gamma))$ are parallel in $M$, then at most $N_M$ of the components  of $\partial (X\cup N(\Gamma))$ are incompressible in $M$.

Thus we assume that some pair of distinct tori $T_i$ and  $T_j$ of $\partial (X\cup N(\Gamma))$ are parallel in $M$.  Then $T_i$ and $T_j$ co-bound a region $R$ in $M$ which is homeomorphic to a product of a torus and an interval.  However, since $T_i$ and $T_j$ are in the characteristic family for $\mathrm{cl}(M-N(\Gamma))$, they cannot be parallel in $\mathrm{cl}(M-N(\Gamma))$.  Thus $R$ intersects $\Gamma$.  But since $\partial R=T_i\cup T_j$ and $\Gamma$ is disjoint from $T_i\cup T_j$, this implies that $\Gamma\subseteq R$.  It now follows that $X\cup N(\Gamma)\subseteq R$.

Now suppose that some component $T_k$ of $\partial(X\cup N(\Gamma))$ is incompressible in $M$ and distinct from $T_i$ and $T_j$.  Then  $T_k\subseteq R$.   But since $R\cong T_j\times I$, either $T_k$ is parallel in $R$ to both  $T_i$ and $T_j$, or $T_k$ is compressible in $R$.  However, since $R\subseteq M$, the latter would imply that $T_k$  is compressible in $M$.  

Thus $T_k$ must be parallel to both $T_i$ and $T_j$ in $R\subseteq M$.  Since $T_k$ does not intersect $\Gamma$,  this means that $T_k$ is parallel to one of $T_i$ or $T_j$ in $\mathrm{cl}(M-N(\Gamma))$.  But this is impossible since $T_i$, $T_j$, and $T_k$ are all in the characteristic family for $\mathrm{cl}(M-N(\Gamma))$.   Thus the only components of $\partial (X\cup N(\Gamma))$ which could be incompressible in $M$ are $T_i$ and $T_j$.  Hence $\partial (X\cup N(\Gamma))$ has at most $2\leq N_M$ components which are incompressible in $M$.  

\bigskip
 
 \noindent{\bf Step 11: We prove the inequality $\mathrm{dim} _{\Z_2} (H_1 (M',\Z_2))\leq n_M.$}
\bigskip

Recall that $M'$ is obtained from $X\cup N(\Gamma)$ by adding a collection of solid tori $V_1$, \dots, $V_r$ along the components $T_1$, \dots, $T_r$ of $\partial (X\cup N(\Gamma))$.  Also, by the condition in Step 9, we know that if $T_i$ is compressible in $M$, then every element of $H_1 (T_i,\mathbb{Z}_2)$ is trivial in $H_1 (M',\mathbb{Z}_2)$; and if $T_i$ is incompressible in $M$ then there exists some curve $\beta_i$ on $T_i$ such that $\beta_i$ is non-trivial in $H_1 (T_i,\mathbb{Z}_2)$ but trivial in $H_1 (M',\mathbb{Z}_2)$.  It follows that for every $T_i$ which is incompressible in $M$, there is at most one element of $H_1 (T_i,\mathbb{Z}_2)$ which is non-trivial in $H_1 (M',\mathbb{Z}_2)$.  Let $\alpha_i$ be a representative of this homology class in $H_1 (T_i,\mathbb{Z}_2)$.  Thus every curve in  $H_1 (T_i,\mathbb{Z}_2)$ is either trivial in $H_1 (M',\mathbb{Z}_2)$ or homologous to $\alpha_i$ in $H_1 (M',\Z_2)$.

Now let $\alpha$ be a curve in $M'$ which is non-trivial in $H_1 (M',\Z_2)$, and for each solid torus $V_i$, let $C_i$ denote its core.  Then by general position, we can choose $\alpha$ to be disjoint from $C_1\cup \dots \cup C_r$, and hence disjoint from $V_1\cup \dots\cup V_r$.  Now since $\alpha\subseteq X\cup N(\Gamma)\subseteq M$, we can also consider $\alpha$ in $ H_1 (M,\Z_2)$.

Suppose that $\alpha$ is trivial in $H_1 (M,\Z_2)$. It follows that $\alpha$ is homologous in $ X\cup N(\Gamma)$ to a collection of curves on $T_1\cup \dots \cup T_r$ which are trivial in $H_1 (M,\Z_2)$ but non-trivial in $H_1 (M',\Z_2)$.   Since for any $T_i$ which is compressible in $M$ every element of $H_1 (T_i,\mathbb{Z}_2)$ is trivial in $H_1 (M',\mathbb{Z}_2)$, $\alpha$ must be homologous in $H_1 (M',\Z_2)$ to a collection of curves on only those $T_i$ which are incompressible in $M$.  It now follows that $\alpha$ is homologous in $H_1 (M',\Z_2)$ to a sum of $\alpha_i$'s on $T_i$'s that are incompressible in $M$.  But by Step 10, there are at most $N_M$ such tori $T_i$.  Hence there are at most $N_M$ distinct $\alpha_i$ which are non-trivial in $H_1 (M',\Z_2)$.  It follows that these $N_M$ curves generate every non-trivial element of $H_1 (M',\Z_2)$ which is trivial in $H_1 (M,\Z_2)$. This gives us the required inequality:

  $$\mathrm{dim} _{\Z_2} (H_1 (M',\Z_2))\leq \mathrm{dim} _{\mathbb{Z}_2} (H_1 (M,\mathbb{Z}_2))+N_M=n_M.$$
  \medskip
  
  Hence the proposition follows.\end{proof}

\bigskip\bigskip

\centerline{\textsc{References cited}}

\bigskip\bigskip

\catcode`\@=11
\def\@biblabel#1{\@ifnotempty{#1}{[\bfseries #1]}}

\renewenvironment{thebibliography}[1]{%
   \normalfont
   \labelsep .5em\relax
   \renewcommand\theenumiv{\arabic{enumiv}}\let\p@enumiv\@empty
   \list{\@biblabel{\theenumiv}}{\settowidth\labelwidth{\@biblabel{#1}}%
     \leftmargin\labelwidth \advance\leftmargin\labelsep
     \usecounter{enumiv}}%
   \sloppy \clubpenalty\@M \widowpenalty\clubpenalty
   \sfcode`\.=\@m
}{%
   \def\@noitemerr{\@latex@warning{Empty `thebibliography'  
environment}}%
   \endlist
}

\newcommand{\bibsub}[1]{${}_{\mathbf{#1}}$}

\catcode`\@=12

\end{document}